\documentclass[a4paper,11pt]{article}
\usepackage[top=3.0cm, bottom=3.0cm, inner=3.0cm, outer=3.0cm, includefoot]{geometry}

\usepackage[utf8]{inputenc}
\usepackage[T1]{fontenc}
\usepackage{authblk}

\usepackage{verbatim}
\usepackage{multirow}
\usepackage{multicol}

\usepackage{amssymb}
\usepackage{amsmath}
\usepackage{mathtools,bbm}
\usepackage{graphicx,tikz}
\usepackage{amsthm}
\usepackage{caption,subcaption} 

\usepackage{color}
\usepackage{enumerate}

\usepackage{cancel}

\newcommand{\grad}{\operatorname{grad}}

\newcommand{\K}{{\mathcal{K}}}

\newcommand{\eChar}{\begin{enumerate}[(i)]}
\newcommand{\eCharR}{\begin{enumerate}[(a)]}
\newcommand{\eBr}{\begin{enumerate}[(1)]}

\newcommand{\Abstract}

\title
{A brief note about $p$-curvature on graphs

}

\author{Chunyang Hu}
\affil{School of Mathematical Sciences, University of Science and Technology of China \\ Hefei 230023, China}
\affil{chunyanghu@mail.ustc.edu.cn}

\date{\today}

\date{\today}

\theoremstyle{plain}
\newtheorem{lemma}{Lemma}[section]

\newtheorem{proposition}[lemma]{Proposition}

\theoremstyle{definition}

\newtheorem{definition}[lemma]{Definition}
\newtheorem{remark}[lemma]{Remark}

\numberwithin{equation}{section}

\begin{document}

\maketitle

\pagestyle{plain}

\begin{abstract}
  In this paper, we consider Wang's $CD_p(m,\K)$ condition on graphs, which depends on the $p$-Laplacian $\Delta_p$ for $p>1$ and is an extension of the classical Bakry-\'Emery $CD(m,\K)$  curvature dimension condition. We calculate several examples including paths, cycles and star graphs, and we show that the $p$-curvature is non-negative at some vertices in the case $p\geq 2$, while it approaches to $-\infty$ in the case of $1<p<2$. In addition, we 
  observe that a crucial property of $\Gamma_2$ on Cartesian products does no longer hold for $\Gamma_2^p$ in the case of $p > 2$. As a consequence, an analogous proof that non-negative curvature is preserved under taking Cartesian products is not possible for $p > 2$.
\end{abstract}

\tableofcontents

%%%%%%%%%%%%%%%%%%%%%%%%%%%%%%%%%%%%%%%%%%%%%%%%%%%%%
\section{Introduction}
%%%%%%%%%%%%%%%%%%%%%%%%%%%%%%%%%%%%%%%%%%%%%%%%%%%%%
In this paper, we present some results on the $p$-Bakry-\'Emery curvature on graphs, introduced by Wang \cite{Wang23}.

%(Brief summary about Bakry-\'Emery curvature)

The original idea behind Bakry-\'Emery curvature comes from 
Bochner's formula, which is the following fundamental pointwise identity in a Riemannian manifold:
$$ \frac{1}{2} \Delta \vert {\rm{grad}} f(x) \vert^2 = \langle {\rm{grad}}\, \Delta f(x), {\rm{grad}}\, f(x) \rangle + \Vert {\rm{Hess}} f(x) \Vert_{HS}^2 + {\rm{Ric}}({\rm{grad}}f(x)). $$
This formula implies the following inequality between the dimension $m$ of the manifold and a lower bound $\mathcal{K}$ of its Ricci curvature at $x \in M$:
$$ \frac{1}{2} \Delta \vert {\rm{grad}} f(x) \vert^2 - \langle {\rm{grad}}\, \Delta f(x), {\rm{grad}}\, f(x) \rangle \ge \frac{1}{m} (\Delta f(x))^2 + \mathcal{K} \vert {\rm{grad}}f(x) \vert^2, $$
which can be reformulated with the help of $\Gamma$-operators in the following way:
\begin{equation} \label{eq:CDmK} 
\Gamma_2 f(x) \ge \frac{1}{m} (\Delta f(x))^2 + \mathcal{K} \Gamma f(x).  
\end{equation}
Here, the $\Gamma$-operators are defined iteratively as follows:
\begin{align*}
\Gamma(f,g)(x) &= \frac{1}{2} \left( \Delta (fg)(x) - f(x) \Delta g(x) - g(x) \Delta f(x) \right), \\
\Gamma_2(f,g)(x) &= \frac{1}{2} \left( \Delta \Gamma(f,g)(x) - \Gamma(f,\Delta g)(x) - \Gamma(g,\Delta f)(x) \right).
\end{align*}
Using graph Laplacians, which are natural linear operators on functions defined on the vertex set $V$, we can also define these $\Gamma$-operators on graphs and say that a vertex satisfies the Bakry-\'Emery curvature-dimension condition $CD(m,\mathcal{K})$, if the inequality \eqref{eq:CDmK} holds for all functions $f: V \to \mathbb{R}$.  
For a given dimension $m$, the optimal value $\mathcal{K}$ is called the ($m$-dimensional) Bakry-\'Emery curvature at the vertex $x$.

In recent years, Bakry-\'Emery curvature on graphs has been widely studied and a lot of interesting results have been derived (see,
 e.g., \cite{CKLP22,LMP24,Siconolfi21} and references therein). For example, Cushing et al. \cite[Theorem 1.2]{CKLP22} and Siconolfi \cite[Theorem 27]{Siconolfi21} independently observed a relation between Bakry-\'Emery curvature and the smallest eigenvalue of a so-called "Curvature matrix", which makes it easier and more direct to compute the curvature at each vertex. In the graph setting, there are also Bonnet-Myer's type diameter bounds for graphs satisfying some strictly positive  Bakry-\'Emery curvature, and the authors in \cite[Proposition 1.3]{LMP24} investigated rigidity properties of this diameter bound.   
 
The $p$-Laplacian $\Delta_p$ is a widely studied non-linear operator, both in the context of Riemannian manifolds and of graphs. For $p>1$, this operator is defined in the manifold setting as
\begin{align*}
    \Delta_pu=\text{div}(|\grad u|^{p-2}\grad u).
\end{align*}
There are many papers concerned with properties of this operator on manifolds. For example, in \cite{KN09}, B. Kotschwar and L. Ni established a sharp Li-Yau gradient estimate for positive solutions to the $p$-Laplace equation on manifolds of non-negative Ricci curvature. 
There is also a well-studied discrete analogue of $\Delta_p$ for $p>1$ on graphs (see Definition \ref{def:plaplacian}). In \cite{HM15}, B. Hua and  Mugnolo studied the nonlinear Cauchy problem of this $p$-Laplacian on an infinite graph.

%(Linfeng Wang's work on $p$-Laplacian and $CD_p$ curvature.)

In \cite{Wang23}, Wang first introduced the definition of classical operators $\Gamma_{2,p}$ (see \cite[Equation (2.5)]{Wang23}) and the $CD_p$ condition(see \cite[Definition 2.2]{Wang23}). The $CD_p$ condition is an analogue of the Bakry-\'Emery curvature-dimension condition depending on $\Delta_p$ instead of $\Delta$. Moreover, in the spirit of M\"unch \cite{Mu17,Mu18}, he showed that the operators $\Delta_p$, $\Gamma_p$ and $\Gamma_{2,p}$ are directional derivatives of some $\psi$-operators $\Delta_p^{\psi}$, $\Gamma_p^{\psi}$ and $\Gamma_{2,p}^{\psi}$, the latter three of which depend on some concave function $\psi:(0,+\infty)\to\mathbb{R}$. He also introduced the $CD^{\psi}_p$ condition and derived a Davies's gradient estimate for positive solutions to the $p$-Laplace parabolic equation on a connected finite graph with $CD_p^{\psi}$ condition and a corresponding Harnack inequality. 

In \cite{XSW24}, X. Xu, W. Shen and L. Wang proved that a $CD_p$ curvature condition is satisfied on any connected locally finite graph for $p\geq 2$, while it does not hold true for $1<p<2$. Moreover, they established a lower bound for the first nonzero eigenvalue of $\Delta_p$ on a connected finite graph using the $CD_p(m,\K)$ condition in the case of $1<p<2$, $m>\frac{2(p-1)^2}{p}$ and $\K>0$. In \cite{Wang20, Wang24,XSW23}, the authors gave a series of analytical results, such as a Liouville theorem and a Harnack inequality and other results, under the $CD^\psi_p$ condition on graphs, in some cases with the special choice of $\psi=\sqrt{\cdot}$.

In the case $p=2$, the $p$-Laplacian and $p$-curvature both reduce to the notions of the standard Laplacian and original Bakry-\'Emery curvature.

%(The main results in this paper.) 

In this paper, we first recall the basic definitions, especially for the operators $\Delta_p$, $\Gamma_p$, $\Gamma_{2,p}$, and Wang's $CD_p(m,\K)$ condition on graphs. Then we consider various examples of paths and star graphs. In the case of $p\geq 2$ and dimension $m=\infty$, we show that $CD_p(\infty,0)$ is satisfied at a middle vertex of $P_3$ (see Proposition \ref{prop:middle_path3}), at a leaf vertex of any path $P_N$ with $N\geq 3$ (see Proposition \ref{prop:leaf_pathn}), and at an arbitrary vertex in the infinite path $P_\infty$ (see Proposition \ref{prop:infinite_path}) or a cycle $C_d$ for $d\geq 5$. In other words, the $CD_p$ curvature (of infinite dimension) at each of the above vertices is non-negative. However, in the case of $1<p<2$, the $CD_p$ curvature at a middle vertex of $P_3$ approaches $-\infty$ surprisingly (see Proposition \ref{prop:negative_curvature}). In the case of a star graph and $p \geq 2$, we derive a precise formula for the $CD_p$ curvature of infinite dimension at any leaf vertex and observe that it is linearly decreasing with respect to the number of total vertices and agrees with the classical Bakry-\'Emery curvature in the case $p=2$ (as verified via the Graph Curvature Calculator, \cite{CKLLS19}). In the last section, we focus on the $CD_p$ curvature on Cartesian product $G_1\times G_2$ of two graphs $G_1=(V_1,E_1)$ and $G_2=(V_2,E_2)$. In the case of $p>2$, we find that, for any two vertices $x\in V_1$ and $y\in V_2$, there exists some function $f:V_1\times V_2 \to \mathbb{R}$, such that the following inequality of $\Gamma_{2,p}$ holds true (see Proposition \ref{prop:negative_pcurvature} below):
\begin{align}\label{ineq:cartesian_product_gamma2}
    \Gamma_{2,p}(f)(x,y)<\Gamma_{2,p}(f^x)(y)+\Gamma_{2,p}(f_y)(x).
\end{align}
Here, $f^x(\cdot):V_2\to \mathbb{R}$, $f_y(\cdot):V_1\to \mathbb{R}$ are defined as $f^x(y)=f_y(x)=f(x,y)$.
However, in the case of $p=2$, the opposite inequality (\ref{ineq:cartesian_product_gamma2}) holds true for any function $f$ (see \cite[Lemma 2.7]{LP18}), and is essential for the proof that Cartesian products of non-negatively curved graphs are, again, non-negatively curved (see \cite[Theorem B]{LP18}). For this reason, an analogous proof for $p$-curvature with $p > 2$ is unfortunately not possible. 

%%%%%%%%%%%%%%%%%%%%%%%%%%%%%%%%%%%%%%%%%%%%%%%%%%%%%
\section{Preliminaries}
%%%%%%%%%%%%%%%%%%%%%%%%%%%%%%%%%%%%%%%%%%%%%%%%%%%%%
Let $G=(V,w,\mu)$ be a weighted graph with the vertex set $V$. We assume that a function $w:V\times V \to [0,+\infty)$ is the edge weight and is symmetric, i.e. $w(u,v)=w(v,u)$ for any $u,v\in V$. We also assume that a  function $\mu:V\to(0,+\infty)$ is the vertex measure. We denote the edge set $E$ as $E:=\{\{u,v\}|w(u,v)=w(u,v)>0\}$. The graph $G$ is undirected according to the symmetry of $w$. For two vertices $u,v\in V$, if there exist vertices $\{x_i\}$ such that $v_0=u\sim v_1 \sim \cdots \sim v_n=v$, we call it a path of length $n$ from $u$ to $v$. Then the combinatorial distance is defined as: for any two vertices $u,v\in V$,
\begin{align*}
    d(u,v):=\inf\{n|\text{There exist a path such that }v_0=u\sim v_1 \sim \cdots \sim v_n=v\}.
\end{align*}
A graph is called connected if for any two vertices, there exists at least one finite path from one to the other. For a vertex $u\in V$, we denote a ball centered at $u$ with radius $i$ as $B_i(u):=\{v\in V|d(u,v)\leq i\}$, and the $i$-th sphere centered at $u$ with radius $i$ as $S_i(u):=\{v\in V|d(u,v)=i\}$. In this paper, we mainly care about the local structure of $2$-ball of a vertex. In particular, for a vertex $u\in V$, we denote $B^{inc}_2(u)$ as the subgraph obtained by deleting the spherical edges on $S_2(u)$ in the induced subgraph of $B_2(u)$. In the following, we introduce three important operators on graphs.

\begin{definition}[$p$-Laplacian]\label{def:plaplacian}
    Let $G=(V,w,\mu)$ be a weighted graph, with a vertex measure $\mu:V\rightarrow{(0,+\infty)}$ and edge weight $w:V\times V\rightarrow{[0,+\infty)}$. For a given constant $p>1$ and any function $u:V\to\mathbb{R}$, the $p$-Laplacian operator $\Delta_{p}$ is defined as: 
    \[\Delta_{p}u(x):=\frac{1}{\mu(x)}\sum_{y:y\sim x}w_{xy}|u(y)-u(x)|^{p-2}(u(y)-u(x)),\]
\end{definition}
In the case of $p=2$, the operator $\Delta_{p}$ reduces to the classical weighted Laplacian $\Delta$: For a function $u:V\to \mathbb{R}$,
\[\Delta u(x):=\frac{1}{\mu(x)}\sum_{y:y\sim x}w_{xy}(u(y)-u(x)).\]

\begin{definition}[$\Gamma_p$-operator]\label{def:gammap}
Let $G=(V,w,\mu)$ be a weighted graph, with a vertex measure $\mu:V\to (0,+\infty)$ and edge weight $w:V\times V\to [0,+\infty)$. For $p>1$ and a fixed function $u:V \to \mathbb{R}$, we define the operator $\Gamma_{p,u}$ as follows: For any vertex $x\in V$ and any two functions $f,g:V\rightarrow{\mathbb{R}}$,
\begin{align}
&\Gamma_{p,u}(f,g)(x):=\notag\\
&\frac{p-1}{2}\frac{1}{\mu(x)}\sum_{y:y\sim x}w_{xy}|u(y)-u(x)|^{p-2}(f(y)-f(x))(g(y)-g(x)).
\end{align}
\end{definition}
For simplicity, we write
\begin{align}
    \Gamma_{p}f(x):=\Gamma_{p,f}(f,f)(x)=\frac{p-1}{2}\frac{1}{\mu(x)}\sum_{y:y\sim x}w_{xy}|f(y)-f(x)|^p.
\end{align}
\begin{definition}[$\Gamma_{2,p}$-operator]\label{def:gamma2p}
Let $G=(V,w,\mu)$ be a weighted graph, with a vertex measure $\mu:V\to (0,+\infty)$ and edge weight $w:V\times V \to [0,+\infty)$. For $p>1$, any vertex $x\in V$ and for any function $f:V\to\mathbb{R}$, we define the operator $\Gamma_{2,p}$ as:
\begin{align}
\Gamma_{2,p}f(x):=&\frac{1}{p(p-1)\mu(x)}\sum_{y:y\sim x}w_{xy}|f(y)-f(x)|^{p-2}(\Gamma_{p}f(y)-\Gamma_{p}f(x))\notag\\
&-\frac{1}{(p-1)^2}\Gamma_{p,f}(f,\Delta_{p}f)(x).
\end{align}
\end{definition}
In the case of $p=2$, these two operators reduce to the classical $\Gamma$- and $\Gamma_2$-operators:
\begin{align}\label{eq:classical_gamma}
\Gamma(f,g)(x):=\frac{1}{2}\frac{1}{\mu(x)}\sum_{y:y\sim x}w_{xy}(f(y)-f(x))(g(y)-g(x))
\end{align}
and
\begin{align}\label{eq:classical_gamma2}
\Gamma_{2}f(x)=\frac{1}{2}\Delta\Gamma(f,f)(x)-\Gamma(f,\Delta f)(x).
\end{align}
Then, we introduce the $CD_p(m,\K)$ condition on graphs depending on these operators.
\begin{definition}[$CD_p(m,\K)$-condition on graphs]
    For $p>1$, $m>0$, $\K\in\mathbb{R}$, we say that a weighted graph $G=(V,w,\mu)$ satisfies the $CD_{p}(m,K)$ condition at all vertices $x\in V$, if for any function $f:V\rightarrow(0,+\infty)$, we have
    \[\Gamma_{2,p}(f)(x)\geq \frac{p-1}{m}(\Delta_{p}f)^2(x)+\K(\Gamma_{p}(f)(x))^{\frac{2p-2}{p}}.\]
    We denote the largest constant $\K$ which satisfies this inequality as $\K_{p,x,G}(m)$.
\end{definition}
In the case of $p=2$, it also reduces to the definition of classical Bakry-\'Emery curvature.
\begin{definition}[Bakry-\'Emery curvature/$CD(m,\K)$-condition]
    Let $G=(V,w,\mu)$ be a weighted graph and $x\in V$ be a vertex. For $\K\in\mathbb{R}$ and $N\in(0,\infty]$, we say it satisfies the $CD(m,\K)$ condition at $x$ if for any function $f:V \to \mathbb{R}$, the following inequality holds true:
    \begin{align}
        \Gamma_{2}(f)(x)\geq \frac{1}{m}(\Delta f(x))^2+\K\Gamma(f)(x).
    \end{align}
    We define the largest constant $\K$ satisfying the above inequality as the $m$-Bakry-\'Emery curvature at this vertex $x\in V$, denoted as $\K_{G,x}(N)$.
\end{definition}

%%%%%%%%%%%%%%%%%%%%%%%%%%%%%%%%%%%%%%%%%%%%%%%%%%%%%
\section{$CD_p(m,\K)$ conditions on Paths}
%%%%%%%%%%%%%%%%%%%%%%%%%%%%%%%%%%%%%%%%%%%%%%%%%%%%%
In this section, we mainly consider about the $p$-curvature on a graph $G=P_N$ or $P_{\infty}$, i.e. a path with finite or infinite length . We consider a middle vertex and a leaf respectively. In this case, we assume that $w_{xy}=1$ for each edge $x\sim y$ and $\mu(x)=1$ for each vertex.

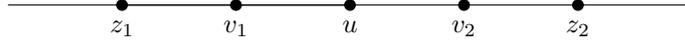
\begin{figure}[!htp]
\centering
\tikzset{vertex/.style={circle, draw, fill=black!20, inner sep=0pt, minimum width=4pt}}
%\tikzstyle{every node}=[circle, draw, fill=black!20, inner sep=0pt, minimum width=4pt]
\begin{tikzpicture}[scale=3.0]

\draw (0,0) -- (-0.5,0) node[midway, below, black]{}
		 -- (-1,0) node[midway, above, black]{}
         -- (0,0) node[midway, above, black]{};
\draw (0,0) -- (0.5,0) node[midway, above, black]{};
\draw (0.5,0) -- (1,0) node[midway, right, black]{};
\draw (1,0) -- (1.5,0) node[midway, right, black]{};
\draw (-1,0) -- (-1.5,0) node[midway, right, black]{};

\node at (0,0) [vertex, label={[label distance=0mm]270: \small $u$}, fill=black] {};
\node at (-0.5,0) [vertex, label={[label distance=0mm]270: \small $v_1$}, fill=black] {};
\node at (0.5,0) [vertex, label={[label distance=0mm]270: \small $v_2$}, fill=black] {};
\node at (-1,0) [vertex, label={[label distance=0mm]270: \small $z_1$}, fill=black] {};
\node at (1,0) [vertex, label={[label distance=0mm]270: \small $z_2$}, fill=black] {};

\end{tikzpicture}
\caption{The graph $G=P_{N}$.}
\label{fig:a_vertex_on_general_path}
\end{figure}

Here is the local structure of a middle vertex $u\in P_{N}$ or $P_{\infty}$. We now begin to calculate the $p$-curvature at $u$. For any function $f:V\rightarrow\mathbb{R}$, we denote 
$$A:=f(v_1)-f(u),\,\,B:=f(v_2)-f(u),\,\,C:=f(z_1)-f(v_1),\,\,D:=f(z_2)-f(v_2),$$
for simplicity.
Then we have functions $\Delta f(u)$ and $\Gamma_{p}f(u)$ as follows, respectively:
\begin{align}
\Delta_{p}f(u)=&\sum_{v:v\sim u}|f(v)-f(u)|^{p-2}(f(v)-f(u))=|A|^{p-2}A+|B|^{P-2}B.\label{eq:p_delta_on_path}\\
\Gamma_{p}f(u)=&\frac{p-1}{2}\sum_{v:v\sim u}|f(v)-f(u)|^{p}=\frac{p-1}{2}(|A|^p+|B|^p).\label{eq:p_gamma_on_path}
\end{align}
Next we compute the function $\Gamma_{2,p}f(u)$:
\begin{align}
\Gamma_{2,p}f(u)=&\frac{1}{p(p-1)}\sum_{v:v\sim u}|f(v)-f(u)|^{p-2}(\Gamma_{p}f(v)-\Gamma_{p}f(u))-\frac{1}{(p-1)^2}\Gamma_{p,f}(f,\Delta_{p}f)(u)\notag\\
=&\frac{1}{2p}\left(|A|^{p-2}|C|^p-|A|^{p-2}|B|^p+|B|^{p-2}|D|^p-|B|^{p-2}|A|^p\right)\notag\\
&+\frac{1}{2(p-1)}(2|A|^{2p-2}+2|B|^{2p-2}+2|A|^{p-2}|B|^{p-2}AB-|A|^{p-2}|C|^{p-2}AC\notag\\
&-|B|^{p-2}|D|^{p-2}BD)\label{eq:p_gamma2_on_path}
\end{align}
Then, we consider there terms in different cases of paths.

\subsection{$CD_p(m,\K)$ condition at the middle vertex on $P_3$ and $P_4$}
In the case of $P_3$, we have the following proposition.
\begin{proposition}\label{prop:middle_path3}
    In the case of $p>2$, for a path $G=P_3$ with length $2$, we assume that $\mu(v)=1$ for each vertex $v\in G$ and $w_{uv}=1$ for each edge $u\sim v$, then it satisfies the $CD_{p}(\infty,0)$ condition at the middle vertex $u\in P_{3}$.
\end{proposition}
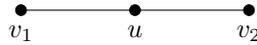
\begin{figure}[!htp]
\centering
\tikzset{vertex/.style={circle, draw, fill=black!20, inner sep=0pt, minimum width=4pt}}
%\tikzstyle{every node}=[circle, draw, fill=black!20, inner sep=0pt, minimum width=4pt]
\begin{tikzpicture}[scale=3.0]

\draw (0,0) -- (-0.5,0) node[midway, below, black]{};
\draw (0,0) -- (0.5,0) node[midway, above, black]{};

\node at (0,0) [vertex, label={[label distance=0mm]270: \small $u$}, fill=black] {};
\node at (-0.5,0) [vertex, label={[label distance=0mm]270: \small $v_1$}, fill=black] {};
\node at (0.5,0) [vertex, label={[label distance=0mm]270: \small $v_2$}, fill=black] {};

\end{tikzpicture}
\caption{The graph $G=P_{3}$.}
\label{fig:middle_vertex_on_p3}
\end{figure}

\begin{proof}
When considering the $p$-curvature at the middle vertex $u\in P_3$, it is equivalent to let $C=D=0$ in the above equations (\ref{eq:p_delta_on_path})-(\ref{eq:p_gamma2_on_path}). That is, $\Delta_pf(u)$ and $\Gamma_pf(u)$ stay put while $\Gamma_{2,p}f(u)$ becomes
\begin{align}\label{eq:p_gamma2}
\Gamma_{2,p}f(u)=\frac{1}{p-1}(|A|^{2p-2}+|B|^{2p-2}+|A|^{p-2}|B|^{p-2}AB)-\frac{1}{2p}(|A|^{p-2}|B|^{p}+|B|^{p-2}|A|^{p}).
\end{align}
In the case of $p>2$, we reformulate the above expression as follows
\begin{align}
    \Gamma_{2,p}f(x)=&\frac{1}{p-1}(\frac{1}{2}|A|^{2p-2}+\frac{1}{2}|B|^{2p-2}+|A|^{p-2}|B|^{p-2}AB)\notag\\
    &+\frac{1}{2(p-1)}(|A|^{2p-2}+|B|^{2p-2})-\frac{1}{2p}(|A|^{p-2}|B|^{p}+|B|^{p-2}|A|^{p})\notag\\
    =&\frac{1}{2(p-1)}(|A|^{p-2}A+|B|^{p-2}B)^2\notag\\
    &+\frac{1}{2(p-1)}(|A|^{p-2}|A|^{p}+|B|^{p-2}|B|^{p}-|A|^{p-2}|B|^{p}-|B|^{p-2}|A|^{p})\notag\\
    &+\left(\frac{1}{2(p-1)}-\frac{1}{2p}\right)(|A|^{p-2}|B|^{p}+|B|^{p-2}|A|^{p})\notag\\
    =&\frac{1}{2(p-1)}(|A|^{p-2}A+|B|^{p-2}B)^2\label{eq:gamma2_first_term}\\
    &+\frac{1}{2(p-1)}((|A|^{p-2}-|B|^{p-2})(|A|^{p}-|B|^{p}))\label{eq:gamma2_second_term}\\
    &+\left(\frac{1}{2(p-1)}-\frac{1}{2p}\right)(|A|^{p-2}|B|^{p}+|B|^{p-2}|A|^{p})\label{eq:gamma2_third_term}.
\end{align}
It is direct to check that each term (\ref{eq:gamma2_first_term}), (\ref{eq:gamma2_second_term}), (\ref{eq:gamma2_third_term}) in the above equation is non-negative for any $A$ and $B$. Hence, $\Gamma_{2,p}f(u)$ is non-negative for any function $f$ in the case of $p>2$. 

Recall that $\Gamma_p(f)(u)$ is non-negative for any function $f:V\to \mathbb{R}$, so is the term $(\Gamma_p(f)(u))^{\frac{2p-2}{p}}$, we can conclude that it satisfies the $CD_{p}(\infty,0)$ condition at the vertex $u\in P_{3}$.
\end{proof}

However, it is different in the case of $1<p<2$. 
\begin{proposition}\label{prop:negative_curvature}
    In the case of $1<p<2$, for a path $G=P_3$ with length $2$, we assume that $\mu(v)=1$ for each vertex $v\in G$ and $w_{uv}=1$ for each edge $u\sim v$, then the $p$-curvature at the middle vertex $u\in P_{3}$ will approach $-\infty$. 
\end{proposition}
\begin{proof}
To see whether the above equation (\ref{eq:p_gamma2}) is non-negative, it is equivalent to check whether the following expression, divided by $|A|^{2p-2}$:
\begin{align}\label{eq:divided_by_p}
    \frac{1}{p-1}\left(1+\frac{|B|^{2p-2}}{|A|^{2p-2}}+\frac{|B|^{p-2}}{|A|^{p-2}}\frac{A}{|A|}\frac{B}{|B|}\right)-\frac{1}{2p}\left(\frac{|B|^p}{|A|^p}+\frac{|B|^{p-2}}{|A|^{p-2}}\right),
\end{align}
is non-negative.
For simplicity, we denote $x:=\frac{B}{|A|}$ and $p=1+a$ for $0<a<1$, then the expression (\ref{eq:divided_by_p}) becomes
\begin{align}
    \frac{1}{p-1}(1+|x|^{2a}+|x|^{a-1}(\pm1)x)-\frac{1}{2p}(x^{1+a}+x^{a-1}).
\end{align}
By letting $x$ approach to $0$, we obtain all the terms except $-\frac{1}{2p}x^{a-1}$ are finite, while this this term approaches to $-\infty$. Therefore, the $p$-curvature at the middle vertex $x\in P_3$ will approach to $-\infty$, in the case of $1<p<2$.
\end{proof}

\begin{proposition}\label{prop:middle_path4}
    In the case of $p>2$, for a path $G=P_4$ with length $3$, we assume that $\mu(v)=1$ for each vertex $v\in G$ and $w_{uv}=1$ for each edge $u\sim v$, then it satisfies the $CD_{p}(\infty,0)$ condition at the middle vertex $u\in P_{3}$.
\end{proposition}
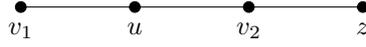
\begin{figure}[!htp]
\centering
\tikzset{vertex/.style={circle, draw, fill=black!20, inner sep=0pt, minimum width=4pt}}
%\tikzstyle{every node}=[circle, draw, fill=black!20, inner sep=0pt, minimum width=4pt]
\begin{tikzpicture}[scale=3.0]

\draw (0,0) -- (-0.5,0) node[midway, below, black]{};
\draw (0,0) -- (0.5,0) node[midway, above, black]{};
\draw (0.5,0) -- (1,0) node[midway, above, black]{};

\node at (0,0) [vertex, label={[label distance=0mm]270: \small $u$}, fill=black] {};
\node at (-0.5,0) [vertex, label={[label distance=0mm]270: \small $v_1$}, fill=black] {};
\node at (0.5,0) [vertex, label={[label distance=0mm]270: \small $v_2$}, fill=black] {};
\node at (1,0) [vertex, label={[label distance=0mm]270: \small $z$}, fill=black] {};

\end{tikzpicture}
\caption{The graph $G=P_{4}$.}
\label{fig:middle_vertex_on_p4}
\end{figure}

\begin{proof}
    Assume $u\in V$ is a middle vertex in $P_4$, the local structure is shown in Figure \ref{fig:middle_vertex_on_p4}. For any function $f:P_4\to \mathbb{R}$, it is equivalent to let $D=0$ in the above equation (\ref{eq:p_delta_on_path})-(\ref{eq:p_gamma2_on_path}). We get $\Delta_pf(u)$ and $\Gamma_pf(u)$ stay put while $\Gamma_{2,p}f(u)$ becomes
    \begin{align*}
    \Gamma_{2,p}f(u)=&\frac{1}{2p}(|A|^{p-2}|C|^p-|A|^{p-2}|B|^p-|B|^{p-2}|A|^p)\\
    &+\frac{1}{2(p-1)}(2|A|^{2p-2}+2|B|^{2p-2}+2|A|^{p-2}|B|^{p-2}AB-|A|^{p-2}|C|^{p-2}AC).
    \end{align*}
    Divided by the non-negative term $|A|^{2p-2}$, we get
    \begin{align*}
        \frac{\Gamma_{2,p}f(x)}{|A|^{2p-2}}=&\frac{1}{2p}\frac{|C|^p}{|A|^p}-\frac{1}{2p}\frac{|B|^p}{|A|^P}-\frac{1}{2p}\frac{|B|^{p-2}}{|A|^{p-2}}\\
        &+\frac{1}{p-1}+\frac{1}{p-1}\frac{|B|^{2p-2}}{|A|^{2p-2}}+\frac{1}{p-1}\frac{|B|^{p-2}}{|A|^{p-2}}\frac{B}{|A|}-\frac{1}{2(p-1)}\frac{|C|^{p-2}}{|A|^{p-2}}\frac{C}{A}
    \end{align*}
    For simplicity, we denote: $x:=\frac{B}{|A|}$, $y:=\frac{C}{|A|}$ and reconsider the above term $\frac{\Gamma_{2,p}f(x)}{|A|^{2p-2}}$ as a function $h(x)$, that is,
    \begin{align*}
        h(x):=\frac{1}{2p}|y|^p-\frac{1}{2p}|x|^{p}-\frac{1}{2p}|x|^{p-2}+\frac{1}{p-1}+\frac{1}{p-1}|x|^{2p-2}+\frac{1}{p-1}|x|^{p-2}x-\frac{1}{2(p-1)}|y|^{p-2}y
    \end{align*}
    Without loss of generality, we can assume that $x \le 0$ and $y\ge 0$. For simplicity, we replace $-x$ by $x$ with the later one satisfying $x\ge 0$. Then we have
    \begin{align*}
        h(x):&=\frac{1}{2p}y^p-\frac{1}{2p}x^p-\frac{1}{2p}x^{p-2}+\frac{1}{p-1}+\frac{1}{p-1}x^{2p-2}-\frac{1}{p-1}x^{p-1}-\frac{1}{2(p-1)}y^{p-1}\\
        &=(\frac{1}{2p}y^p-\frac{1}{2(p-1)}y^{p-1})+(\frac{1}{p-1}x^{2p-2}-\frac{1}{2p}x^{p-2}-\frac{1}{p-1}x^{p-1}-\frac{1}{2p}x^p)+\frac{1}{p-1}
    \end{align*}
    Since the variables $x$ and $y$ are independent, then we consider the terms involving $y$ first. We can get 
    $$\min_{y \ge 0}(\frac{1}{2p}y^p-\frac{1}{2(p-1)}y^{p-1}) = \frac{1}{2p}-\frac{1}{2(p-1)},$$
    by differentiating with respect to $y$. Replacing the terms by its minimum yields
    \begin{align*}
    h(x)&=\frac{1}{p-1}x^{2p-2}-\frac{1}{2p}x^{p-2}-\frac{1}{p-1}x^{p-1}-\frac{1}{2p}x^p+\frac{1}{2p}+\frac{1}{2(p-1)}\\
    &=\frac{1}{2p}(x^{2p-2}-x^{p-2}-x^p+1)+\frac{p+1}{2p(p-1)}x^{2p-2}-\frac{1}{p-1}x^{p-1}+\frac{1}{2(p-1)}\\
    &=\frac{1}{2p}(x^p-1)(x^{p-2}-1)+\frac{1}{2(p-1)}\left((x^{p-1}-1)^2+\frac{1}{p}x^{2p-2}\right) \ge 0.
    \end{align*}
    Therefore, we conclude that it satisfies $CD_p(\infty,0)$ condition at a middle vertex on the path $P_4$, in the case of $p>2$.
\end{proof}
Again, different result occurs in the case of $1<p<2$.
\begin{proposition}\label{prop:negative_curvature}
    In the case of $1<p<2$, for a path $G=P_4$ with length $3$, we assume that $\mu(v)=1$ for each vertex $v\in G$ and $w_{uv}=1$ for each edge $u\sim v$, then the $p$-curvature at the middle vertex $u\in P_{3}$ will approach $-\infty$. 
\end{proposition}
\begin{proof}
    By a similar argument with the proof for Proposition \ref{prop:middle_path4}, we obtain the function $h(x)$ as follows:
    \begin{align*}
        h(x)=\frac{1}{p-1}x^{2p-2}-\frac{1}{2p}x^{p-2}-\frac{1}{p-1}x^{p-1}-\frac{1}{2p}x^p+\frac{1}{2p}+\frac{1}{2(p-1)}.
    \end{align*}
    Since now $p$ satisfies $1<p<2$, we have $p-1>0$ and $p-2<0$. Letting $x \to 0$ yields $x^{p-2} \to \infty$ while other terms are finite. Therefore, $h(0)$ goes to $-\infty$ and this completes the proof.
\end{proof}

\subsection{$CD_p(m,\K)$ conditions at a leaf on $P_N$}
In this subsection, we consider the $p$-curvature at a leaf vertex on a path $P_{N}$ for an integer $N\geq2$.
\begin{figure}[!htp]
\centering
\tikzset{vertex/.style={circle, draw, fill=black!20, inner sep=0pt, minimum width=4pt}}
%\tikzstyle{every node}=[circle, draw, fill=black!20, inner sep=0pt, minimum width=4pt]
\begin{tikzpicture}[scale=3.0]

\draw (0,0) -- (-0.5,0) node[midway, below, black]{};
\draw (0,0) -- (0.5,0) node[midway, above, black]{};
\draw (0.5,0) -- (1,0) node[midway, above, black]{};

\node at (0,0) [vertex, label={[label distance=0mm]270: \small $v$}, fill=black] {};
\node at (-0.5,0) [vertex, label={[label distance=0mm]270: \small $u$}, fill=black] {};
\node at (0.5,0) [vertex, label={[label distance=0mm]270: \small $z$}, fill=black] {};

\end{tikzpicture}
\caption{The graph $G=P_{N}$.}
\label{fig:a_leaf_on_a_path}
\end{figure}
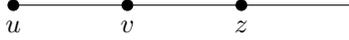
\begin{proposition}\label{prop:leaf_pathn}
In the case of $p>1$, for a path $G=P_N$ with length $N\geq2$, we assume that $\mu(v)=1$ for each vertex $v\in G$ and $w_{uv}=1$ for each edge $u\sim v$, then it satisfies the $CD_{p}(\infty,0)$ condition at a leaf vertex $u\in P_{N}$.
\end{proposition}
\begin{proof}
It is equivalent to let $B=D=0$ in the above equation (\ref{eq:p_gamma2_on_path}) for $\Gamma_{2,p}f(u)$. Then we obtain 
\begin{align}
 \Gamma_{2,p}f(x)=\frac{1}{p-1}|A|^{2p-2}+\frac{1}{2p}|A|^{p-2}|C|^p-\frac{1}{2(p-1)}|A|^{p-2}|C|^{p-2}AC.       
\end{align}
Again, we divide it by $|A|^{2p-2}$ and we have
\begin{align*}
\frac{\Gamma_{2,p}f(x)}{|A|^{2p-2}}=\frac{1}{p-1}+\frac{1}{2p}\frac{|C|^p}{|A|^p}-\frac{1}{2(p-1)}\frac{|C|^{p-2}}{|A|^{p-2}}\frac{C}{|A|}.
\end{align*}
We regard it as a function $g(x)$ of $x:=\frac{C}{|A|}$ as the following:
\begin{align}
    g(x)=\frac{1}{p-1}+\frac{1}{2p}|x|^p-\frac{1}{2(p-1)}|x|^{p-2}x
\end{align}
In the case of $x\leq0$, i.e. $C\leq 0$, it always holds that $g(x)
\geq 0$. So we only need to consider the case $x>0$, i.e. $C>0$. In this case, we have
\begin{align*}
    g(x)=\frac{1}{p-1}+\frac{1}{2p}x^p-\frac{1}{2(p-1)}x^{p-1}.
\end{align*}
At the point $x=0$, we obtain $g(x)=\frac{1}{p-1}>0$. Taking the derivative yields
\[g'(x)=\frac{1}{2}x^{p-1}-\frac{1}{2}x^{p-2}=\frac{1}{2}x^{p-2}(x-1).\]
The zero points of $g'(x)$ are $x=0$ and $x=1$. And we know that this function $g(x)$ is decreasing only in the interval $(0,1)$ so it attains its minimal value as the point $x=1$, whose value is given by $g(1)=\frac{1}{2p}+\frac{1}{2(p-1)}>0$. Therefore, $\Gamma_{2,p}f(u)$ is always non-negative for any function $f$.

Since $\Gamma_p(f)(u)$ is non-negative for any function $f$ and any vertex $u \in G$, so is the term $(\Gamma_p(f)(u))^{\frac{2p-2}{p}}$. In  conclusion, it satisfies the $CD_{p}(\infty,0)$ condition at a leaf vertex $u\in P_{N}$ for $N\geq2$.
\end{proof}

\subsection{$CD_{p}(m,\K)$ condition on a general path $P_{N}$ or $P_{\infty}$}

In this subsection, we pay attention to a middle vertex in a general path $P_N$ with $N\geq 5$, shown in Figure \ref{fig:a_vertex_on_general_path}. Here comes our result.
\begin{proposition}\label{prop:infinite_path}
    In the case of $p>2$, for a general path $G=P_N$ with $N\geq 5$ or $G=P_{\infty}$, we assume that $\mu(v)=1$ for each vertex $v\in G$ and $w_{uv}=1$ for each edge $u\sim v$, then it satisfies the $CD_{p}(\infty,0)$ condition at any middle vertex $u\in G$.
\end{proposition}
\begin{proof}
    From the equation (\ref{eq:p_gamma2_on_path}) above, we reformulate it as
    \begin{align}\label{eq:gamma2p_path}
        \Gamma_{2,p}f(u)=&\frac{1}{p-1}(|A|^{2p-2}+|B|^{2p-2})+\frac{1}{p-1}|A|^{p-2}|B|^{p-2}AB-\frac{1}{2p}(|A|^{p-2}|B|^p+|B|^{p-2}|A|^{p})\\
        &+\left(\frac{1}{2p}|A|^{p-2}|C|^{p}-\frac{1}{2(p-1)}|A|^{p-2}|C|^{p-2}AC\right)\\
        &+\left(\frac{1}{2p}|B|^{p-2}|D|^{p}-\frac{1}{2(p-1)}|B|^{p-2}|D|^{p-2}BD\right).
    \end{align}
    Again, to see whether it is non-negative is equivalent to divide it by $|A|^{2p-2}$. That is,
    \begin{align*}
        \frac{\Gamma_{2,p}f(u)}{|A|^{2p-2}}=&\frac{1}{p-1}\left(1+\frac{|B|^{2p-2}}{|A|^{2p-2}}\right)+\frac{1}{p-1}\left(\frac{|B|^{p-2}}{|A|^{p-2}}\frac{A}{|A|}\frac{B}{|A|}\right)-\frac{1}{2p}\left(\frac{|B|^p}{|A|^p}+\frac{|B|^{p-2}}{|A|^{p-2}}\right)\\
        &+\frac{1}{2p}\left(\frac{|C|^P}{|A|^p}\right)-\frac{1}{2(p-1)}\left(\frac{|C|^{p-2}}{|A|^{p-2}}\frac{A}{|A|}\frac{C}{|A|}\right)\\
        &+\frac{1}{2p}\left(\frac{|B|^{p-2}}{|A|^{p-2}}\frac{|D|^p}{|A|^p}\right)-\frac{1}{2(p-1)}\left(\frac{|B|^{p-2}}{|A|^{p-2}}\frac{|D|^{p-2}}{|A|^{p-2}}\frac{B}{|A|}\frac{D}{|A|}\right).
    \end{align*}
    For simplicity, we denote $x:=\frac{B}{|A|}$, $y:=\frac{C}{|A|}$ and $z:=\frac{D}{|A|}$. Than the expression becomes as follows:
    \begin{align}
    \frac{\Gamma_{2,p}f(x)}{|A|^{2p-2}}=&\frac{1}{p-1}(1+|x|^{2p-2})+\frac{1}{p-1}(|x|^{p-2}(\pm1)x)-\frac{1}{2p}(|x|^{p}+|x|^{p-2})\notag\\
    &+\left(\frac{1}{2p}|y|^p-\frac{1}{2(p-1)}(|y|^{p-2}(\pm1)y)\right)\notag\\
    &+\left(\frac{1}{2p}|x|^{p-2}|z|^p-\frac{1}{2(p-1)}(|x|^{p-2}|z|^{p-2}xz)\right).\label{eq:gamma2_x_y_z}
    \end{align}
    Without loss of generality, we assume that all the variables $x$, $y$ and $z$ are non-negative, then we can get rid of absolute values in the equation (\ref{eq:gamma2_x_y_z}) and it  becomes as the follow:
    \begin{align}
        \frac{\Gamma_{2,p}f(u)}{|A|^{2p-2}}&=\frac{1}{p-1}(1+x^{2p-2})+\frac{1}{p-1}x^{p-1}-\frac{1}{2p}(x^p+x^{p-2})\label{eq:gamma2_x}\\
        &+\left(\frac{1}{2p}y^p-\frac{1}{2(p-1)}y^{p-1}\right)\label{eq:gamma2_y}\\
        &+\left(\frac{1}{2p}x^{p-2}z^p-\frac{1}{2(p-1)}x^{p-1}z^{p-1}\right)\label{eq:gamma2_z}.
    \end{align}
    Since these three variables $x$, $y$ and $z$ are independent with each other, than we can consider the expressions in the above three lines respectively, and (\ref{eq:gamma2_y}) at first.
    
    Take the term (\ref{eq:gamma2_y}) as a function of $y$ and denote it as: 
    \begin{align}\label{eq:function_g}
    g(y):=\frac{1}{2p}y^p-\frac{1}{2(p-1)}y^{p-1}, \,\,\text{for }y\geq0.
    \end{align}
    Since $g(0)=0$, we than take the derivative and it gives
    \[g'(y)=\frac{1}{2p}py^{p-1}-\frac{1}{2(p-1)}(p-1)y^{p-2}=\frac{1}{2}y^{p-2}(y-1).\]
    Hence, $g'(y)$ is non-positive on the interval $[0,1]$ and non-negative in the interval $[1,+\infty)$. Therefore, $g(y)$ is decreasing in $[0,1]$, increasing in $[1,+\infty)$ and will attains its minimum at $y=1$, that is, 
    \begin{align}\label{eq:g_y_minimum}
    g(y)_{\min}=g(1)=\frac{1}{2p}-\frac{1}{2(p-1)}=-\frac{1}{2p(p-1)}.
    \end{align}

    For the expression (\ref{eq:gamma2_z}), we regard $z$ as a variable and $x$ as a parameter, and we define:
    \[h(z):=\frac{1}{2p}x^{p-2}z^{p}-\frac{1}{2(p-1)}x^{p-1}z^{p-1},\,\,\text{for }x\geq0\,\,\text{and }z\geq0.\]
    We know that $h(0)=0$ and the derivative of $h(z)$ is
    \[h'(z)=\frac{1}{2p}x^{p-2}(pz^{p-1})-\frac{1}{2(p-1)}x^{p-1}(p-1)z^{p-2}=\frac{1}{2}x^{p-2}z^{p-2}(z-x).\]
    It is similar to see that $h(z)$ is decreasing in the interval $[0,x]$ and increasing on the interval $[x,\infty)$. Hence, $h(z)$ attains its minimum at the point $z=x$ and the minimal value is 
    \begin{align}\label{eq:h_z_minumum}
    h(z)_{min}=h(x)=-\frac{1}{2p(p-1)}x^{2p-2}.
    \end{align}
    Next, we replace the terms (\ref{eq:gamma2_y}) and (\ref{eq:gamma2_z}) with their minimum values, respectively. That is,
    \begin{align*}
        \frac{\Gamma_{2,p}f(u)}{|A|^{2p-2}}=&\frac{1}{p-1}(1+x^{2p-2})+\frac{1}{p-1}x^{p-1}-\frac{1}{2p}(x^p+x^{p-2})-\frac{1}{2p(p-1)}-\frac{1}{2p(p-1)}x^{2p-2}.
    \end{align*}
    We reformulate it as the following:
    \begin{align*}
       \frac{\Gamma_{2,p}f(u)}{|A|^{2p-2}}=&\frac{1}{2(p-1)}(1+x^{p-1})^2+\frac{1}{2(p-1)}(1+x^{2p-2})-\frac{1}{2p}(x^p+x^{p-2})\\
       &-\frac{1}{2p(p-1)}-\frac{1}{2p(p-1)}x^{2p-2}\\
       =&\frac{1}{2(p-1)}(1+x^{p-1})^2+\frac{1}{2p}(1+x^{2p-2}-x^p-x^{p-2})\\
       =&\frac{1}{2(p-1)}(1+x^{p-1})^2+\frac{1}{2p}(x^{p-2}-1)(x^{p}-1)\geq 0.
    \end{align*}
    According to the last expression, it is obvious to see that $\Gamma_{2,p}f(u)$ is non-negative for any function $f$ on graphs in the case $p>2$. 
    
    Since $\Gamma_p(f)(u)$ is non-negative for any function $f:V\to\mathbb{R}$ and for any vertex $u\in G$, so is the term $(\Gamma_p(f)(u))^{\frac{2p-2}{2}}$. We can derive that it satisfies the $CD_{p}(\infty,0)$ condition at any vertex $x\in G$ on a path.
\end{proof}

\section{$CD_p(m,\K)$ condition on cycles}
In this section, we study the $CD_p(m,\K)$ condition on cycles.
\begin{proposition}
    For $p>2$, let $G=(V,w,\mu)$ be a cycle $C_d$ with $d\geq 5$. Assume that $\nu(v)=1$ for each vertex $v\in V$ and $w_{uv}=1$ for each edge $u \sim v$, then it satisfies the $CD_p(\infty,0)$ condition at each vertex $u\in V$.
\end{proposition}
\begin{proof}
    For a cycle $C_d$ with $d\geq 5$ and any vertex $u\in V$, the incomplete local structure $B^{inc}_2(u)$ of $u\in V$ is the same as that in a path, as shown in Figure \ref{fig:a_vertex_on_general_path}. Therefore, we can derive the conclusion directly from Proposition \ref{prop:infinite_path}. 
\end{proof}
\begin{proposition}
    For $p>2$, let $G=(V,w,\mu)$ be a $4$-cycle $C_4$. Assume that $\mu(v)=1$ for each vertex $v\in V$ and $w_{uv}=1$ for each edge $u\sim v$, then it satisfies the $CD_p(\infty,0)$ condition at each vertex $u\in V$.
\end{proposition}
\begin{proof}
    For any vertex $x\in V$, the local structure $B^{inc}_2(u)$ is shown in Figure \ref{fig:c4}.
    
    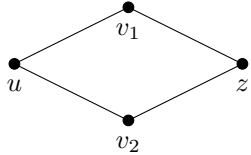
\begin{figure}[!htp]
\centering
\tikzset{vertex/.style={circle, draw, fill=black!20, inner sep=0pt, minimum width=4pt}}
%\tikzstyle{every node}=[circle, draw, fill=black!20, inner sep=0pt, minimum width=4pt]
\begin{tikzpicture}[scale=3.0]

\draw (0,0.25) -- (-0.5,0) node[midway, below, black]{};
\draw (0,0.25) -- (0.5,0) node[midway, above, black]{};
\draw (0,-0.25) -- (-0.5,0) node[midway, below, black]{};
\draw (0,-0.25) -- (0.5,0) node[midway, above, black]{};

\node at (0,0.25) [vertex, label={[label distance=0mm]270: \small $v_1$}, fill=black] {};
\node at (-0.5,0) [vertex, label={[label distance=0mm]270: \small $u$}, fill=black] {};
\node at (0.5,0) [vertex, label={[label distance=0mm]270: \small $z$}, fill=black] {};
\node at (0,-0.25) [vertex, label={[label distance=0mm]270: \small $v_2$}, fill=black] {};

\end{tikzpicture}
\caption{The local structure at $x$ in cycle graph $G=C_4$.}
\label{fig:c4}
\end{figure}
For any function $f:V\to\mathbb{R}$, we compute the terms $\Gamma_p(f)(u)$ and $\Gamma_{2,p}(f)(u)$, respectively.
\begin{align*}
    \Gamma_p(f)(u)=\frac{p-1}{2}(|f(v_1)-f(u)|^p+|f(v_2)-f(u)|^p)\geq 0.
\end{align*}
And similarly, we use the notation as follows:
$$A:=f(v_1)-f(u),\,\,B:=f(v_2)-f(u),\,\,C:=f(z)-f(v_1),\text{ and }D:=f(z)-f(v_2).$$
Then
\begin{align}
    \Gamma_{2,p}(f)(u)=&\frac{1}{2p}|A|^{p-2}|C|^p-\frac{1}{2p}|A|^{p-2}|B|^p+\frac{1}{2p}|B|^{p-2}|D|^P-\frac{1}{2p}|B|^{p-2}|A|^p\notag\\
    &+\frac{1}{p-1}|A|^{2p-2}-\frac{1}{2(p-1)}|A|^{p-2}A|C|^{p-2}C+\frac{1}{p-1}|A|^{p-2}A|B|^{p-2}B\notag\\
    &+\frac{1}{p-1}|B|^{2p-2}-\frac{1}{2(p-1)}|B|^{p-2}B|D|^{p-2}D.\label{eq:gamma2p_cycle4}
\end{align}
It is obvious that the formula (\ref{eq:gamma2p_cycle4}) of $\Gamma_{2,p}(f)(x)$ on $C_4$ is precisely the same as the formula (\ref{eq:gamma2p_path}) on a long path. However, notice that in the case of $C_4$, the four variables $A$, $B$, $C$ and $D$ are no longer independent with each other. More specifically, they satisfy the following equation:
\begin{equation*}
  S_1:=\{A,B,C,D| C-B=f(z)-f(y_1)-f(y_2)+f(x)=D-A.\}
\end{equation*}
Since now the condition $S_1$ is only a subset of 
$$S:=\{A,B,C,D|A,B,C,D\in\mathbb{R}\text{ are arbitrary.}\},$$
and in the latter case $\Gamma_{2,p}(f)(u)$ is non-negative, we can derive directly that 
$$\Gamma_{2,p}(f)(u)\geq 0.$$
Recall that $\Gamma_p(f)(u)$ is always non-negative, we can get the conclusion that it satisfies $CD_p(\infty,0)$ condition at any vertex in the cycle $C_4$.
\end{proof}
\begin{remark}
    For $1<p<2$, let $G=(V,w,\mu)$ be a triangle $C_4$. Assume that $\mu \equiv 1$ and $w\equiv 1$, then it does NOT satisfy the $CD_p(\infty, 0)$ condition. This example is shown in \cite[Remark 3.1]{XSW24}. Moreover, the $p$-curvature at $u\in C_4$ will approach $-\infty$.
\end{remark}
Finally, we deal with the smallest cycle $C_3$.
\begin{proposition}
    For $p>2$, let $G=(V,w,\mu)$ be a triangle $C_3$. Assume that $\mu(x)=1$ for each vertex $v\in V$ and $w_{uv}=1$ for each edge $u \sim v$, then it satisfies the $CD_p(\infty,0)$ condition at each vertex $u\in V$.
\end{proposition}
\begin{proof}
    For any vertex $u\in V$, the local structure $B^{inc}_2(u)$ is shown in Figure \ref{fig:c3}.

\begin{figure}[!htp]
\centering
\tikzset{vertex/.style={circle, draw, fill=black!20, inner sep=0pt, minimum width=4pt}}
%\tikzstyle{every node}=[circle, draw, fill=black!20, inner sep=0pt, minimum width=4pt]
\begin{tikzpicture}[scale=3.0]

\draw (0,0.25) -- (-0.5,0) node[midway, below, black]{};
\draw (0,0.25) -- (0,-0.25) node[midway, above, black]{};
\draw (0,-0.25) -- (-0.5,0) node[midway, below, black]{};

\node at (0,0.25) [vertex, label={[label distance=0mm]0: \small $v$}, fill=black] {};
\node at (-0.5,0) [vertex, label={[label distance=0mm]270: \small $u$}, fill=black] {};
\node at (0,-0.25) [vertex, label={[label distance=0mm]0: \small $z$}, fill=black] {};

\end{tikzpicture}
\caption{The local structure at $x$ in triangle $G=C_3$.}
\label{fig:c3}
\end{figure}
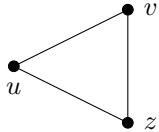

For any function $f:V\to \mathbb{R}$, we have
\begin{align*}
    \Gamma_p(f)(u)=\frac{p-1}{2}(|f(v)-f(u)|^p+|f(z)-f(u)|^p)\geq 0.
\end{align*}
And we denote:
$$A:=f(v)-f(u),\,\,B:=f(z)-f(u),\text{ and }C:=f(z)-f(v),$$
\begin{align}
    \Gamma_{2,p}(f)(u)=&\frac{1}{2p}|A|^{p-2}|C|^p+\frac{1}{2p}|B|^{p-2}|C|^p-\frac{1}{2p}|A|^{p-2}|B|^{p}-\frac{1}{2p}|A|^p|B|^{p-2}\notag\\
    &+\frac{1}{p-1}|A|^{2p-2}+\frac{1}{p-1}|B|^{2p-2}+\frac{1}{p-1}|A|^{p-2}A|B|^{p-2}B\notag\\
    &-\frac{1}{2(p-1)}|A|^{p-2}A|C|^{p-2}C+\frac{1}{2(p-1)}|B|^{p-2}B|C|^{p-2}C.\label{eq:gamma2p_c3}
\end{align}
It is direct to check that the above formula (\ref{eq:gamma2p_c3}) is same as the formula (\ref{eq:gamma2p_path}) by choosing $D=-C$. And in this case, there is one more relation between these variables. Specifically, the conditions now for the variables are:
$$S_2=\{A,B,C,D|D=-C,\text{ and }B-A=C\}.$$
The condition $S_2$ is again a subset of $S=\{A,B,C,D\}$. Therefore, we can derive that in a triangle, 
$$\Gamma_{2,p}(f)(u)\geq 0$$
for any function $f$ and it satisfies the $CD_p(\infty,0)$ condition on a triangle $C_3$, with the non-negativity of the term $\Gamma_p(f)(u)$. We've finished the proof.
\end{proof}

%%%%%%%%%%%%%%%%%%%%%%%%%%%%%%%%%%%%%%%%%
\section{$CD_{p}(m, \K)$ condition on a star graph}
%%%%%%%%%%%%%%%%%%%%%%%%%%%%%%%%%%%%%%%%%
In this section, we consider the $CD_{p}(m,\K)$ condition on a star graph $S_{\ell+1}$, here $\ell+1$ is the total number of leaves,  as shown in Figure \ref{fig:a_star_graph} below.

\begin{figure}[!htp]
\centering
\tikzset{vertex/.style={circle, draw, fill=black!20, inner sep=0pt, minimum width=4pt}}
%\tikzstyle{every node}=[circle, draw, fill=black!20, inner sep=0pt, minimum width=4pt]
\begin{tikzpicture}[scale=3.0]

\draw (0,0) -- (-0.5,0) node[midway, below, black]{};
\draw (0,0) -- (0.5,0.5) node[midway, above, black]{};
\draw (0,0) -- (0.5,0.25) node[midway, above, black]{};
\draw (0,0) -- (0.5,-0.25) node[midway, above, black]{};
\draw (0,0) -- (0.5,-0.5) node[midway, above, black]{};

\node at (0,0) [vertex, label={[label distance=0mm]270: \small $v$}, fill=black] {};
\node at (-0.5,0) [vertex, label={[label distance=0mm]270: \small $u$}, fill=black] {};
\node at (0.5,0.5) [vertex, label={[label distance=0mm]0: \small $z_1$}, fill=black] {};
\node at (0.5,0.25) [vertex, label={[label distance=0mm]0: \small $z_2$}, fill=black] {};
\node at (0.5,-0.25) [vertex, label={[label distance=0mm]0: \small $z_{\ell-1}$}, fill=black] {};
\node at (0.5,-0.5) [vertex, label={[label distance=0mm]0: \small $z_\ell$}, fill=black] {};
\node at (0.5,0) [vertex, label={[label distance=0mm]0: \small $\vdots$}, fill=black] {};

\end{tikzpicture}
\caption{A general star graph $G=S_{\ell+1}$}
\label{fig:a_star_graph}
\end{figure}
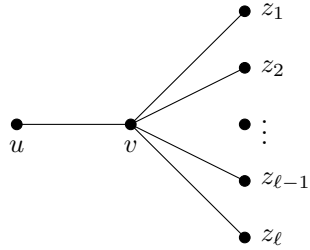
\begin{proposition}
    In the case of $p\geq 2$ and dimension $m=\infty$, for a star graph $G=S_{\ell+1}$, we assume that the vertex measure $\mu(u)=1$ for each vertex $u\in G$ and the edge measure $w_{uv}=1$ for each edge $u\sim v$. Then the $p$-curvature at a leaf vertex $u\in S_{\ell+1}$ is:
    \[\K_{p,u,G}(\infty)=\frac{4}{(p-1)^{2}}\left(\frac{p-1}{2}\right)^{\frac{2}{p}}\left(\frac{1}{p-1}-\frac{\ell-1}{2p(p-1)}\right).\]
   In addition, this $p$-curvature $\K_{p,u,G}(\infty)$ is decreasing linearly with respect to $\ell$. Moreover, if $\ell > 2p + 1$, we have that the $p$-curvature $\K_{p,u,G}(\infty)$ is negative.
\end{proposition}
\begin{proof}
    For any function $f:V\rightarrow \mathbb{R}$, we can calculate $\Gamma_pf(u)$ and $\Gamma_{2,p}f(u)$ respectively. For simplicity, we first denote
    \[A:=f(v)-f(u),\,\,B_{1}:=f(z_1)-f(v),\,\,\cdots,B_{\ell}:=f(z_{\ell})-f(v).\]
    Then, we have
    \[\Gamma_{p}f(u)=\frac{p-1}{2}\sum_{v:v\sim u}|f(v)-f(u)|^p=\frac{p-1}{2}|f(v)-f(u)|^p=\frac{p-1}{2}|A|^{p}.\]
    And
    \begin{align*}
        \Gamma_{2,p}f(u)=&\frac{1}{2p}\left(\sum_{i:i=1}^{\ell-1}|f(v)-f(u)|^{p-2}|f(z_{i}-f(v))|^p\right)+\frac{1}{p-1}(|f(v)-f(u)|^{2p-2})\\
        &-\frac{1}{2(p-1)}\sum_{i:i=1}^{\ell-1}|f(v)-f(u)|^{p-2}|f(z_i)-f(v)|^{p-2}|f(v)-f(u)||f(z_i)-f(u)|\\
        =&\frac{1}{p-1}|A|^{2p-2}+\frac{1}{2p}\sum_{i:i=1}^{\ell-1}|A|^{p-2}|B_i|^{p}-\frac{1}{2(p-1)}\sum_{i:i=1}^{\ell-1}|A|^{p-2}|B_i|^{p-2}AB.
    \end{align*}
    In the case of dimension $N=\infty$, $p$-curvature is the largest constant $\K_{p,u,G}(\infty)$ such that for any function $f$, the following inequality holds true:
    \[\Gamma_{2,p}f(u)\geq \K_{p,u}(\infty) (\Gamma_{p}f(u))^{\frac{2p-2}{p}}.\]
    It is equivalent to 
    \[K_{p,u}(\infty)=\inf_{f:V\rightarrow \mathbb{R},f\not\equiv c}\frac{\Gamma_{2,p}f(u)}{(\Gamma_{p}f(u))^{\frac{2p-2}{p}}}.\]
    We can derive
    \begin{align*}
        \K_{p,u}(\infty)=&\inf_{f:V\rightarrow \mathbb{R},f\not\equiv c}\frac{4}{(p-1)^2}\left(\frac{p-1}{2}\right)^{\frac{2}{p}}\left(\frac{1}{p-1}+\sum_{i:i=1}^{\ell-1}\left(\frac{1}{2p}\frac{|B_i|^{p}}{|A|^{p}}-\frac{1}{2(p-1)}\frac{|B_i|^{p-2}}{|A|^{p-2}}\frac{B_i}{|A|}\right)\right).
    \end{align*}
    In a similar way, we denote $y_i:=\frac{B_i}{|A|}$ for simplicity and without loss of generality, we assume $y_i\geq 0$ for all $i=1,\cdots,l.$ That is, 
      \begin{align*}
        \K_{p,u}(\infty)=&\inf_{f:V\rightarrow \mathbb{R},f\not\equiv c}\frac{4}{(p-1)^2}\left(\frac{p-1}{2}\right)^{\frac{2}{p}}\left(\frac{1}{p-1}+\sum_{i:i=1}^{\ell-1}\left(\frac{1}{2p}y_i^p-\frac{1}{2(p-1)}y_i^{p-1}\right)\right).
    \end{align*}
    For each term in the summation, since these variables $\{y_i\}_i$ are independent with each other, we can regard each term as a function $g_i$ of $y_{i}$, $(y_i\geq 0)$. According to the arguments (\ref{eq:function_g})-(\ref{eq:g_y_minimum}) above, we can derive the minimum of $\K_{p,x}(\infty)$ as:
    \[\K_{p,u}(\infty)=\frac{4}{(p-1)^2}\left(\frac{p-1}{2}\right)^{\frac{2}{p}}\left(\frac{1}{p-1}-\frac{\ell-1}{2p(p-1)}\right).\]
    In addition, the $p$-curvature $\K_{p,u,G}(\infty)$ at a leaf $u\in G$ decreases linearly with respect to the number of vertices $\ell+1$.
\end{proof}
\begin{remark}
    In the case of $p=2$, the above formula becomes
    \begin{align*}
        \K_{u}(\infty)=2-\frac{l-1}{2},
    \end{align*}
    which precisely corresponds to the curvature of a leaf vertex in a star graph, as shown in Graph Curvature Calculator.
\end{remark}

%%%%%%%%%%%%%%%%%%%%%%%%%%%%%%%%%%%%%%%%%%%%%%%%%%%%%
\section{$p$-curvature on Cartesian product}
%%%%%%%%%%%%%%%%%%%%%%%%%%%%%%%%%%%%%%%%%%%%%%%%%%%%%
In this section, we focus on the $p$-curvature on Cartesian product. First we give the definition of Cartesian product of two graphs.
\begin{definition}
    Let $G_1=(V_1,E_1)$ and $G=(V_2,E_2)$ be two graphs. The Cartesian product, denoted as $G_1\times G_2$, of $G_1$ and $G_2$ are defined as follows: The vertex set is 
    $$V_{G_1\times G_2}:=\{(x,y)|x\in V_1,y\in V_2\}=V_1\times V_2,$$
    and the edge set is defined as:
    $$E_{G_1\times G_2}:=\{(x_1,x_2)\sim(y_1,y_2)|\text{ if }x_1=y_1\text{ and }x_2\sim y_2,\text{ or }x_1\sim x_2 \text{ and }y_1=y_2\}.$$
\end{definition}
Suppose $x\in V_1$ and $y\in V_2$, we denote the neighbors of $x$ and $y$ are:
$$S_1(x)=\{x_i\}_{i=1}^{d_x},\,\,\,S_1(y)=\{y_{k}\}_{k=1}^{d_y},$$
respectively. We also denote vertices in $2$-spheres of $x$ and $y$ as:
$$S_2(x)=\{x_{ij}\}_{ij},\,\,\,S_2(y)=\{y_{kl}\}_{kl}.$$
We denote the local structure of incomplete $2$-ball as $B_2^{inc}(x,y)$, which is obtained by deleting the edges in $2$-sphere in the induced subgraph of $B_2(x,y)$. The following picture Figure \ref{fig:cartesian_product} shows the local structure of $B^{inc}_2(x,y)$:
\begin{figure}[!htp]
\centering
\tikzset{vertex/.style={circle, draw, fill=black!20, inner sep=0pt, minimum width=4pt}}
%\tikzstyle{every node}=[circle, draw, fill=black!20, inner sep=0pt, minimum width=4pt]
\begin{tikzpicture}[scale=3.0]

\draw (0,0) -- (1,0.5) node[midway, below, black]{};
\draw (0,0) -- (1,-0.5) node[midway, above, black]{};
\draw (1,0.5) -- (2,0) node[midway, above, black]{};
\draw (1,-0.5) -- (2,0) node[midway, above, black]{};
\draw (1,0.5) -- (2,1) node[midway, above, black]{};
\draw (1,-0.5) -- (2,-1) node[midway, above, black]{};

\node at (0,0) [vertex, label={[label distance=0mm]270: \small $(x,y)$}, fill=black] {};
\node at (1,0.5) [vertex, label={[label distance=0mm]110: \small $\{x_i,y\}_i$}, fill=black] {};
\node at (1,-0.5) [vertex, label={[label distance=0mm]250: \small $\{x,y_k\}_k$}, fill=black] {};
\node at (2,1) [vertex, label={[label distance=0mm]0: \small $\{x_{ij},y\}_{ij}$}, fill=black] {};
\node at (2,0) [vertex, label={[label distance=0mm]0: \small $\{x_i,y_k\}_{i,k}$}, fill=black] {};
\node at (2,-1) [vertex, label={[label distance=0mm]0: \small $\{x,y_{kl}\}_{kl}$}, fill=black] {};

\end{tikzpicture}
\caption{The Cartesian product $G_1\times G_2$ of $G_1$ and $G_2$}
\label{fig:cartesian_product}
\end{figure}
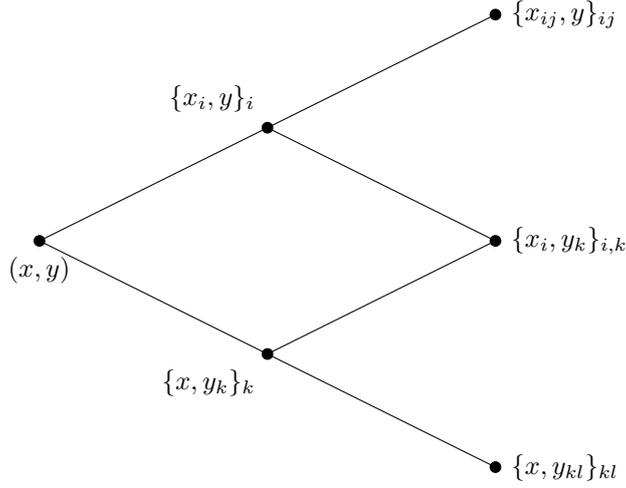

Here comes our main result.
\begin{proposition}\label{prop:negative_pcurvature}
  In the case of $p>2$, suppose $G_1\times G_2$ is the Cartesian product of two graphs $G_1=(V_1,E_1)$ and $G_2=(V_2,E_2)$. Assume that edge weight satisfies $w\equiv1$ and vertex measure $\mu\equiv 1$. Then there exists a function $f:V_1\times V_2 \to \mathbb{R}$, such that the following inequality holds true:
  \begin{align}\label{eq:cartesian_product}
  \Gamma_{2,p}(f)(x,y)< \Gamma_{2,p}(f^x)(y)+\Gamma_{2,p}(f_y)(x).
  \end{align}
  Here we write $f^x(\cdot):=f(x,\cdot)$ as a function on $V_2$, for a fixed $x\in V_1$, and similarly, $f_y(\cdot):=f(\cdot,y)$ as a function on $V_1$, for some fixed vertex $y\in V_2$.
\end{proposition}
\begin{proof}
   On the Cartesian product $G_1\times G_2$ for two graphs $G_1=(V_1,E_1)$ and $G_2=(V_2,E_2)$, for any function $f:V_1\times V_2 \to \mathbb{R}$ and any vertex $(x,y)\in V_1\times V_2$, we first have the additivity property for both $\Delta_{p}(f)$ and $\Gamma_{p}(f)$:
   \begin{align}
       \Delta_{p}f(x,y)=&\sum_{(u,v)\sim(x,y)}|f(u,v)-f(x,y)|^{p-2}(f(u,v)-f(x,y))\notag\\
       =&\sum_{x_i\sim x}|f(x_i,y)-f(x,y)|^{p-2}(f(x_i,y)-f(x,y))\notag\\
       &+\sum_{y_k\sim y}|f(x,y_k)-f(x,y)|^{p-2}(f(x,y_k)-f(x,y))\notag\\
       =&\Delta_{p}(f_y)(x)+\Delta_p(f^x)(y)\label{eq:cartesian_product_laplacian}
   \end{align}
   and
   \begin{align}
       \Gamma_p(f)(x,y)=&\frac{p-1}{2}\sum_{(u,v)\sim (x,y)}|f(u,v)-f(x,y)|^p\notag\\
       =&\frac{p-1}{2}\sum_{x_i\sim x}|f(x_i,y)-f(x,y)|^p+\frac{p-1}{2}\sum_{y_k\sim y}|f(x,y_k)-f(x,y)|^p\notag\\
       =&\Gamma_p(f_y)(x)+\Gamma_p(f^x)(x).\label{eq:cartesian_product_gammap}
   \end{align}
   Then here comes $\Gamma_{2,p}(f)(x,y)$ by definition:
   \begin{align}
       \Gamma_{2,p}(f)(x,y)=&\frac{1}{p(p-1)}\sum_{(u,v)\sim(x,y)}|f(u,v)-f(x,y)|^{p-2}(\Gamma_p(f)(u,v)-\Gamma_P(f)(x,y))\notag\\
       &-\frac{1}{(p-1)^2}\Gamma_{p,f}(f,\Delta_p f)(x,y)\label{eq:cartesian_product_gamma2}
   \end{align}
   We calculate each term above carefully. Respectively denote $T_1$ and $T_2$ as follows:
   \begin{align}
       T_1+T_2:=&\frac{1}{p(p-1)}\sum_{x_i\sim x}|f(x_i,y)-f(x,y)|^{p-2}(\Gamma_p(f)(x_i,y)-\Gamma_p(f)(x,y))\\
       &+\frac{1}{p(p-1)}\sum_{y_k\sim y}|f(x,y_k)-f(x,y)|^{p-2}(\Gamma_p(f)(x,y_k)-\Gamma_p(f)(x,y)).
   \end{align}
   According to the additivity of $\Gamma_{p}(f)(x,y)$, i.e. (\ref{eq:cartesian_product_gammap}), we have 
   \begin{align}
       T_1=&\frac{1}{p(p-1)}\sum_{x_i\sim x}|f(x_i,y)-f(x,y)|^{p-2}(\Gamma_p(f_y)(x_i)+\Gamma_p(f^{x_i})(y)-\Gamma_p(f^x)(y)-\Gamma_p(f_y)(x))\notag\\
       =&:L_1+\frac{1}{p(p-1)}\sum_{x_i \sim x}|f(x_i,y)-f(x,y)|^{p-2}(\Gamma_p(f^{x_i})(y)-\Gamma_p(f^x)(y))\notag\\
       =&L_1+\frac{1}{2p}\sum_{\substack{x_i\sim x\\y_k \sim y}}|f(x_i,y)-f(x,y)|^{p-2}(|f(x_i,y_k)-f(x_i,y)|^p-|f(x,y_k)-f(x,y)|^p),\label{eq:T1-L1}
   \end{align}
   here $L_1$ is precisely one of the terms which appears in $\Gamma_{2,p}(f_y)(x)$.
   
   Similarly, we compute the term $T_2$ and denote $L_2$ accordingly:
   \begin{align}
       T_2=&\frac{1}{p(p-1)}\sum_{y_k\sim y}|f(x,y_k)-f(x,y)|^{p-2}(\Gamma_p(f^x)(y_k)+\Gamma_p(f_{y_k})(x)-\Gamma_p(f^x(y))-\Gamma_p(f_y)(x))\notag\\
       =&:L_2+\sum_{y_k \sim y}|f(x,y_k)-f(x,y)|^{p-2}(\Gamma_p(f_{y_k})(x)-\Gamma_p(f_{y})(x))\notag\\
       =&L_2+\frac{1}{2p}\sum_{\substack{x_i\sim x\\y_k\sim x}}|f(x,y_k)-f(x,y)|^{p-2}(|f(x_i,y_k)-f(x,y_k)|^p-|f(x_i,y)-f(x,y)|^p).\label{eq:T2-L2}
   \end{align}
   For the last term in (\ref{eq:cartesian_product_gamma2}), we have
   \begin{align}
       &\frac{1}{(p-1)^2}\Gamma_{p,f}(f,\Delta_p f)(x,y)\notag\\
       =&\frac{1}{2(p-1)}\left(\sum_{x_i \sim x}|f(x_i,y)-f(x,y)|^{p-2}(f(x_i,y)-f(x,y))(\Delta_p f(x_i,y)-\Delta_p f(x,y))\right)\\
       &+\frac{1}{2(p-1)}\left(\sum_{y_k\sim y}|f(x,y_k)-f(x,y)|^{p-2}(f(x,y_k)-f(x,y))(\Delta_p f(x,y_k)-\Delta_p f(x,y))\right)\\
       =&:T_3+T_4.\notag
   \end{align}
   For $T_3$, by the additivity of $\Delta_p$, (see (\ref{eq:cartesian_product_laplacian})), we get
   \begin{align}
       T_3=&\frac{1}{2(p-1)}\sum_{x_i \sim x}|f(x_i,y)-f(x,y)|^{p-2}(f(x_i,y)-f(x,y))(\Delta_p (f^{x_i})(y)+\Delta_p(f_y)(x_i)\notag\\
       & -\Delta_p(f^x)(y)-\Delta_p(f_y)(x))\notag\\
       =&:L_3+\frac{1}{2(p-1)}\sum_{x_i \sim x}|f(x_i,y)-f(x,y)|^{p-2}(f(x_i,y)-f(x,y))(\Delta_p(f^{x_i})(y)-\Delta_p(f^x)(y))\notag\\
       =&L_3+\frac{1}{2(p-1)}\sum_{\substack{x_i\sim x\\y_k\sim y}}|f(x_i,y)-f(x,y)|^{p-2}(f(x_i,y)-f(x,y))(|f(x_i,y_k)-f(x_i,y)|^{p-2}\notag\\
       &(f(x_i,y_k)-f(x_i,y))-|f(x,y_k)-f(x,y)|^{p-2}(f(x,y_k)-f(x,y))).
   \end{align}
   Here $L_3$ is precisely one of the terms appearing in $\Gamma_{2,p}(f_y)(x)$.
   Similarly, we have for the term $T_4$:
   \begin{align}
       T_4=&:L_4+\frac{1}{2(p-1)}\sum_{\substack{x_i \sim x\\y_k \sim y}}|f(x,y_k)-f(x,y)|^{p-2}(f(x,y_k)-f(x,y))(|f(x_i,y_k)-f(x,y_k)|^{p-2}\notag\\
       &(f(x_i,y_k)-f(x,y_k))-|f(x_i,y)-f(x,y)|^{p-2}(f(x_i,y)-f(x,y)))
   \end{align}
   For simplicity, we again introduce several notations for these differences.
   \begin{align*}
   A_i:&=f(x_i,y)-f(x,y),\,\,B_k:=f(x,y_k)-f(x,y),\\
   C_{ik}:&=f(x_i,y_k)-f(x_i,y),\,\,D_{ik}:=f(x_i,y_k)-f(x,y_k).
   \end{align*}
   Notice that these four variables are not totally independent with each other, even $f$ can be arbitrarily chosen. The relation between them is 
   \begin{equation}\label{eq:relationABCD}
       C_{ik}-B_{k}=D_{ik}-A_i,\,\,\text{ for all }1\leq i \leq |S_1(x)| \text{ and }1\leq k \leq |S_1(y)|.
   \end{equation}
   Then we compute the difference between $\Gamma_{2,p}(f)(x,y)$ and $\Gamma_{2,p}(f^x)(y)$ and $\Gamma_{2,p}(f_y)(x)$ as follows:
   \begin{align}
       &\Gamma_{2,p}(f)(x,y)-\Gamma_{2,p}(f^x)(y)-\Gamma_{2,p}(f_y)(x)\label{eq:difference_gamma2}\\
       &=(T_1+T_2-T_3-T_4)-(L_1-L_3)-(L_2-L_4)\notag\\
       &=(T_1-L_1)+(T_2-L_2)-((T_3-L_3)+(T_4-L_4))\notag\\
       &=\sum_{\substack{x_i \sim x\\y_k \sim y}}\left(\frac{1}{2p}|A_i|^{p-2}|C_{ik}|^p-\frac{1}{2p}|A_i|^{p-2}|B_k|^p+\frac{1}{2p}|B_k|^{p-2}|D_{ik}|^p-\frac{1}{2p}|A_{i}|^p|B_{k}|^{p-2}\right)\notag\\
       &-\sum_{\substack{x_i\sim x\\y_k\sim y}}\left(\frac{1}{2(p-1)}|A_i|^{p-2}A_i|C_{ik}|^{p-2}C_{ik}-\frac{1}{2(p-1)}|A_i|^{p-2}A_i|B_k|^{p-2}B_k\right)\notag\\
       &-\sum_{\substack{x_i \sim x\\y_k \sim y}}\left(\frac{1}{2(p-1)}|B_k|^{p-2}B_k|D_{ik}|^{p-2}D_{ik}-\frac{1}{2(p-1)}|A_i|^{p-2}A_i|B_{k}|^{p-2}B_k\right).
   \end{align}
   If we choose $A_i=1,B_k=0,C_{ik}=1$ and $D_{ik}=2$ for each $i$ and each $k$, then they satisfies the above equation (\ref{eq:relationABCD}) and the difference (\ref{eq:difference_gamma2}) is $\sum_{\substack{x_i \sim x\\y_k\sim y}}\frac{1}{2p}-\frac{1}{2(p-1)}=\sum_{\substack{x_i \sim x\\y_k\sim y}}-\frac{1}{2p(p-1)}<0$, then the above inequality (\ref{eq:cartesian_product}) holds.
\end{proof}
\begin{remark}
    In the case of $p=2$, in which operators come back to the classic Laplacian $\Delta$ and gamma-calculus: $\Gamma$
 and $\Gamma_2$, the difference (\ref{eq:difference_gamma2}) becomes
 \begin{align}
     &\Gamma_{2}(f)(x,y)-\Gamma_2(f^x)(y)-\Gamma_2(f_y)(x)\notag\\
     &=\sum_{\substack{x_i \sim x\\y_k \sim y}}\left(\frac{1}{4}|C_{ik}|^2-\frac{1}{4}|B_k|^2+\frac{1}{4}|D_{ik}|^2-\frac{1}{4}|A_i|^2-\frac{1}{2}A_iC_{ik}+\frac{1}{2}A_iB_k-\frac{1}{2}B_kD_{ik}+\frac{1}{2}A_iB_{k}\right)\notag\\
     &=\sum_{\substack{x_i \sim x\\y_k \sim y}}\frac{1}{4}(|C_{ik}|^2-|B_k|^2+|D_{ik}|^2-|A_{i}|^2-2A_i(C_{ik}-B_k)-2B_k(D_{ik}-A_i))\notag\\
     &=\sum_{\substack{x_i \sim x\\y_k \sim y}}\frac{1}{4}(|C_{ik}|^2-|B_k|^2+|D_{ik}|^2-|A_i|^2-2A_i(D_{ik}-A_{i})-2B_k(C_{ik}-B_{k}))\notag\\
     &=\sum_{\substack{x_i \sim x\\y_k \sim y}}\frac{1}{4}((A_i-D_{ik})^2+(C_{ik}-B_{k})^2)\geq 0.
 \end{align}
 Here in the third equality, we apply the equality (\ref{eq:relationABCD}) to interchange the variables twice. Hence, we have the following inequality:
 \begin{equation}
     \Gamma_2(f)(x,y)\geq \Gamma_2(f^x)(y)+\Gamma_2(f_y)(x),
 \end{equation}
  which is proved in \cite[Lemma 2.5]{LP18}. Therefore, in this case, it can be derived a Bakry-\'Emery curvature condition on Cartesian product of two graphs $G_1=(V_1,E_1)$ and $G_2=(V_2,E_2)$ as follows:
 $$\K_{G_1\times G_2}(x,y)\geq \min\{\K_{G_1}(x),\K_{G_2}(y)\}.$$
 \end{remark}

%%%%%%%%%%%%%%%%%%%%%%%%%%%%%%%%%%%%%%%%%%%%%%%%%%%%%

\section*{Acknowledgement}
The author is grateful to the hospitality of Durham University. The author is also very grateful to Professor Norbert Peyerimhoff for making me aware of this topic and valuable discussions and to Professor Shiping Liu for his helpful comments. This work is supported by the "Outstanding PhD Students Overseas Study Program of the University of Science and Technology of China".

\end{document}